\documentclass[submission,copyright,creativecommons]{eptcs}
\providecommand{\event}{} 
\usepackage{breakurl}             

\usepackage{amsmath}
\usepackage{amsthm}
\usepackage{tikz}
\usepackage{extpfeil}
\usepackage{tikz-cd}

\usepackage{amsmath}
\usepackage{tikz}
\usepackage{extpfeil}
\usepackage{tikz-cd}
\usetikzlibrary{snakes}
\usepackage{physics}

\newextarrow{\xbigtoto}{{20}{20}{20}{20}}
   {\bigRelbar\bigRelbar{\bigtwoarrowsleft\rightarrow\rightarrow}}
\newextarrow{\xbigto}{{20}{20}}
   {\bigRelbar{\bigtwoarrowsleft\rightarrow}}
   
\usepackage{caption}
\usepackage{subcaption}

\usepackage{appendix}
\usepackage{hyperref}
\hypersetup{
    colorlinks=true,
    linkcolor=blue,
    filecolor=magenta,
    urlcolor=cyan,
}

\newtheorem{theorem}{Theorem}
\newtheorem{definition}[theorem]{Definition}
\newtheorem{proposition}[theorem]{Proposition}
\newtheorem{corollary}[theorem]{Corollary}

\title{Characterization of Contextuality with Semi-Module Čech Cohomology and its Relation with Cohomology of Effect Algebras}
\author{Sidiney B. Montanhano 
\institute{Departamento de Matemática Aplicada\\
  Instituto de Matemática, Estatística e Computação Científica\\
  Universidade Estadual de Campinas\\
  Campinas, São Paulo, Brazil}
  \email{\texttt{\href{mailto:s226010@dac.unicamp.br}{s226010@dac.unicamp.br}}}}

\def\titlerunning{Contextuality and Semi-Module Cohomology}
\def\authorrunning{S. B. Montanhano}
\begin{document}
\maketitle
\hypersetup{urlcolor=magenta}

\begin{abstract}
I present a generalized notion of obstruction in Čech cohomology on semi-modules, which allows one to characterize non-disturbing contextual behaviors with any semi-field. This framework generalizes the usual Čech cohomology used in the sheaf approach to contextuality, avoiding false positives. I revise a similar work done in the framework of effect algebras with cyclic and order cohomologies and explore the relationship with the one presented here.
\end{abstract}

\section{Introduction}

With the codification of contextuality in the presheaf formalism \cite{Abramsky_2011}, one can use cohomological structures to explore possible detection of contextual behaviour \cite{Abramsky_2012,abramsky_et_al:LIPIcs:2015:5416,Aasn_ss_2020}. But the use of Abelian groups is not justified by the empirical model structure, and violations of this characterization occur \cite{Car__2017,caru2018towards}. To avoid these violations, a framework based on the algebraic structure of empirical models is necessary, implying in the study of cohomology on semi-modules \cite{patchkoria2006exactness,Jun_2017}. For the characterization, however, we need to use the notion of obstruction, which is absent in this generalization of cohomology. In this paper, I introduce a generalized notion of obstruction to codify the data of the difference between contexts and achieve the characterization of contextual behavior with this generalized notion of cohomology.

Once the result is presented, which follows in analogy with the previous use of cohomology for identification of contextual behavior, usually in the logical sense, I revise the formalism of effect algebras and the use of cohomological tools to identify contextuality \cite{Roumen_2017}. The relation between effect algebras and empirical models is close, as one can see in Ref. \cite{10.1007/978-3-662-47666-6_32}, and can be traced to the relation between different notions of contextuality \cite{Wester_2018}. I show that both cohomological frameworks, even looking so different in their constructions, look for the description of the model as a Boolean structure.

This paper is divided as follows. In section \ref{Presheaf}, it's given a revision of contextuality and the presheaf approach. In section \ref{Čech} the cohomology for empirical models and some already known results from the literature are presented. Section \ref{What} uses the categorical structure to show how the usual cohomology forgets data from the empirical model, enabling violations, and examples of it are given. Section \ref{semi-modules} is a naive revision on cohomology using semi-modules instead of groups. In section \ref{Obstruction} the description of the generalized obstruction and how it characterizes $R$-contextual behavior is obtained, presenting the main result of this work. Section \ref{Effect} revise the effect algebra approach and the cohomologies used to detect contextual behavior, and its relation with the framework presented in this paper is explored. Section \ref{Commentaries} presents revised examples and some commentaries about cohomological contextuality and other methods to find contextual behavior, and in section \ref{Conclusion} the conclusions and future paths are given.

\section{Presheaf approach}
\label{Presheaf}

Contextuality can be informally understood as the impossibility of a global description of local parts of a whole. A framework to formalize this idea is the sheaf theory, more specifically the presheaf structure. 

\begin{definition}
A presheaf is a functor $F:C^{op}\to \textbf{Set}$ of a category $C$ to the category of sets. Let $C$ be a site, a small category equipped with a coverage $J$, in other words any object $U\in C$ admits a collection of families of morphisms ${f_{i}:U_{i}\to U}_{i\in I}$ called covering families. A presheaf on $(C,J)$ is a sheaf if it satisfies the following axioms
\begin{itemize}
    \item Gluing: if for all $i\in I$ we have $s_{i}\in F(U_{i})$ such that $s_{i}|_{U_{i}\cap U_{j}}=s_{j}|_{U_{i}\cap U_{j}}$, then there is $s\in F(U)$ satisfying $s_{i}=s|_{U_{i}}$;
    \item Locality: if $s,t\in F(U)$ such that $s|_{U_{i}}=t|_{U_{i}}$ for all $U_{i}$, then $s=t$.
\end{itemize}
\end{definition}

For a presheaf, one of the axioms that define a sheaf fails. And this is what characterizes the sheaf approach to contextuality. Let's define the structure of this approach.

\subsubsection*{Measurement scenario}

The sheaf approach to contextuality has as fundamental objects measurements, organized as a covering through compatibility. To the covering of measurements and the possible events, we give the name measurement scenario \cite{Abramsky_2018}.

\begin{definition}
A measurement scenario $\left<X,\mathcal{U},(O_{x})_{x\in X}\right>$ is a hypergraph $\left<X,\mathcal{U}\right>$, that will be called a scenario, plus some additional structure, usually enough to identify it as a simplicial complex\footnote{See Ref. \cite{Sidiney_2021} for a justified construction of the measurement scenario.}. $X$ is the set of measurements, and $\mathcal{U}$ a covering of contexts, a family of sets of compatible measurements\footnote{A formal definition of compatibility can be found in Ref. \cite{Guerini_2017}. A set of measurements $\mathcal{A}$ is compatible if there is a measurement that can recover all the measurements of $\mathcal{A}$ with marginalization and classical post-processing.}. The sets $(O_{x})_{x\in X}$ for each  $x\in X$ are called outcome sets, with their elements the possible events of each measurement.
\end{definition}

In this paper, we will suppose that outcome sets are finite, and therefore one can define an outcome set $O$ for all the measurements $x$, and codify any $O_{x}$ through an injective function $f_{x}:O_{x}\to O$. We just need to ignore elements that aren't in the image of $f_{x}$, such that these elements aren't in the measure's support. We will work with the set of events of contexts $(O_{U})_{U\in\mathcal{U}}$, and this construction will be implied. We will also work with a measurement scenario with a simplicial complex structure of contexts.

\subsubsection*{Sheaf of events}

We can formalize the structure of events as a sheaf. First, the covering $\mathcal{U}$ will be restricted here to maximal contexts, that will be denoted also by $\mathcal{U}$, so that $\left<X,\mathcal{U}\right>$ can be understood as a set $X$ with a covering $\mathcal{U}$ of maximal contexts $U_{j\in I}$ indexed by an ordered set $I$ \footnote{Given a covering, one can construct with unions and intersections a locale, a "pointless space". It means that the measurements aren't the fundamental objects, turning the minimal contexts into the "effective" measurements of the scenario, which depend on how one chooses the covering. A physical example of refinement is spin degeneration, where refinement occurs by applying a suitable magnetic field.}. As the intersection of contexts is a context, we can define the inclusion morphism $\rho(jk,j): U_{j}\cap U_{k}\to U_{j}$, which turns the set of contexts and the inclusion morphisms in a small category.\footnote{From Ref. \cite{Johnstone:592033} we can see that the category of contexts with the inclusion is a site.} Therefore we can understand a scenario as a small category.

\begin{definition}
The outcome sets are defined by a functor $\mathcal{E}:\left<X,\mathcal{U}\right>^{op}\to\textbf{Set}$, with $\mathcal{E}::U\mapsto O^{U}$ and $\mathcal{E}::\rho\mapsto\rho'$, such that for each element $U\in\mathcal{U}$ we have an outcome set $O^{U}$ of the context and $\rho'$ the restriction to the outcome sets.
\end{definition}

These outcome spaces are by definition the tensor product $O^{U}=\bigotimes_{x\in U}O_{x}$, representing all the combined outcomes. The morphisms of the  $\left<X,\mathcal{U}\right>$ are reversed as restrictions $\rho(j,kj):U_{j}\to U_{j}\cap U_{k}$ where the intersection $U_{j}\cap U_{k}$ is denoted by the indices $jk$, and they are mapped to $\rho'(j,kj):O^{j}\to O^{kj}::s_{j}\mapsto s_{j}|_{kj}$. Elements $s\in O^{U}$ of the image of a sheaf, and of a presheaf, are called local sections if $U\neq X$, and global sections if $U=X$. 

A compatible family is a family of sections $\left\{s_{i}\in\mathcal{E}(i)\right\}_{i\in I}$ such that for all $j$ and $k$ holds $\rho'(j,jk)(s_{j})=\rho'(k,jk)(s_{k})\in\mathcal{E}(jk)$. The Gluing axiom says that the existence of a compatible family implies the existence of a global section $s\in\mathcal{E}(X)$ that satisfies $s|_{j}=s_{j}$ for all $j$. The locality axiom says that if two global sections agree on all the elements of the covering, then they are the same. For the outcomes, these axioms can be interpreted as the independence and uniqueness of events concerning contexts, and we will assume they hold.

\begin{proposition}
By independence and uniqueness of events, the functor $\mathcal{E}$ is a sheaf in the site of measurements and contexts, called the sheaf of events of a given measurement scenario.
\end{proposition}

\subsubsection*{$R$-empirical model}

$R$-empirical models are defined usually with a semi-ring $R$, such as the Boolean semi-ring $\mathbb{B}$, the reals $\mathbb{R}$, or the probability semi-ring $\mathbb{R}^{+}$. The choice of an $R$ defines a way to probe the model. 

\begin{definition}
A generic semi-ring is a set $R$ equipped with two binary operations $+$ and $\cdot$ such that $(R,+)$ is a commutative monoid with $0_{R}$:
\begin{itemize}
    \item $(a+b)+c=a+(b+c)$
    \item $0_{R}+a=a+0_{R}=a$
    \item $a+b=b+a$
\end{itemize}
$(R,\cdot)$ is a monoid with $1_{R}$:
\begin{itemize}
    \item $(a\cdot b)\cdot c=a\cdot (b\cdot c)$
    \item $a\cdot 1_{R}=1_{R}\cdot a=a$
\end{itemize}
the multiplication distributes:
\begin{itemize}
    \item $a\cdot (b+c)=a\cdot b +a\cdot c$
    \item $(a+b)\cdot c=a\cdot c+b\cdot c$
\end{itemize}
and $0_{R}$ annihilates:
\begin{itemize}
    \item $a\cdot 0_{R}=0_{R}\cdot a=0_{R}$.
\end{itemize}
\end{definition}

To define $R$-empirical models, we use another functor $\mathcal{D}_{R}:\textbf{Set}\to\textbf{Set}::O^{U}\mapsto\left\{\mu^{O^{U}}_{R}\right\}$, taking a set of local events to the set $R$-measures defined on it $\mu^{O^{U}}_{p}:\mathbb{P}\left(O^{U}\right)\to R$ that satisfies $\mu^{O^{U}}_{R}(O^{U})=1_{R}$, in analogy with probabilistic measure. We will denote by $\mu_{p}::U\in\mathcal{U}\mapsto\mu_{p}^{O^{U}}$ a set of $R$-measures defined in each element of $\mathcal{U}$, and call it a state. In the morphisms $\mathcal{D}_{R}::\rho'(j,kj)\mapsto \rho''(j,kj)$, with $\rho''(j,kj)::\mu^{O^{j}}_{R}\mapsto\mu^{O^{j}|_{kj}}_{R}=\mu^{O^{j}}_{R}|_{kj}$ the marginalization of the $R$-measure $j$ on the intersection $kj$. 

\begin{definition}
The tuple $(X,\mathcal{U},\mathcal{E},\mu_{R})=e_{R}$ is called an $R$-empirical model over the measurement scenario $\left<X,\mathcal{U},(O_{x})_{x\in X}\right>=\left(\left<X,\mathcal{U}\right>,\mathcal{E}\right)$ given by the state $\mu_{R}$, defining a set of local sections $\left\{\mu_{R}^{O^{U}}\in\mathcal{D}_{R}\mathcal{E}(U);U\in\mathcal{U}\right\}$.
\end{definition}

\subsubsection*{Non-disturbance}

The non-disturbance condition is a usual condition imposed in an $R$-empirical model, sometimes implicitly. It says that $\mu^{O^{j}}_{R}|_{kj}=\mu^{O^{k}}_{R}|_{kj}$ for all $k$ and $j$, which means there is the local agreement between contexts. This condition is equivalent to the existence of a compatible family to $\mathcal{D}_{R}\mathcal{E}$, but it doesn't imply in $\mathcal{D}_{R}\mathcal{E}$ to be a sheaf. Once we can only have access to contexts, it is possible to define the functor $\mathcal{D}_{R}\mathcal{E}$ through a state that can't be extended to a measure in the global events.

Non-disturbance is equivalent to the notion of parameter-independence, as explained in Ref. \cite{barbosa2019continuousvariable}, a property that if violated means the existence of non-trivial data between contexts \cite{Sidiney_2021}. As said in Ref. \cite{Dzhafarov_2018}, where disturbance is called inconsistent connectedness: "Intuitively, inconsistent connectedness is a manifestation of direct causal action of experimental set-up upon the variables measured in it". In this paper, we will work with non-disturbing models.

\subsubsection*{Contextuality}

Contextuality is the impossibility of describe a given $R$-empirical model in classical terms, but one must first define which classical notion to use. We will call it $R$-contextuality to make explicit the chosen semi-ring. First we know that any measure can be described as the marginalization of another one
\begin{equation}
    \mu_{R}^{O^{U}}(A)=\sum_{\lambda\in\Lambda}k\left(\lambda,A\right),
\end{equation}
for all $A\in\mathbb{P}(O^{U})$, $k:\Lambda\times\mathbb{P}(O^{U})\to R$ being a $R$-measure that satisfies $\sum_{\lambda\in\Lambda}k\left(\lambda,O^{U}\right)=1_{R}$. In literature of contextuality and non-locality $\Lambda$ is called the set of hidden variables, which is statistically taken into account but is empirical inaccessible. 

To impose a classical behavior, the hidden variables must be independent of the contexts, a property called $\lambda$-independence\footnote{$\lambda$-independence is related with the concept of free choice in non-locality \cite{Cavalcanti_2018,Abramsky_2014}. It can be understood as a dependence of the hidden variables, sometimes called ontic variables, in the contexts. Such dependence can storage contextuality, as free choice can be understand as storage of non-locality \cite{Blasiak_2021}. For more details of classification of hidden variables, in the subject of non-locality, see Ref. \cite{Brandenburger_2008}.}. To reflect such a behavior of independence, our model must show independence between measurements, in other words be factorizable. Such independence allow to write
\begin{equation}
    \mu_{R}^{O^{U}}(A)=\sum_{\Lambda}p\left(\lambda,A\right)\prod_{x\in U}\mu_{R}^{O^{x}}(\rho'(U,x)(A)),
\end{equation}
with the auxiliary of the set of hidden variables $\Lambda$ being statistically taken into account by a measure $p:\Lambda\to R$ satisfying $p(\Lambda)=1_{R}$. Summing up it with $\lambda$-independence implies
\begin{equation}
    \mu_{R}^{O^{U}}(A)=\sum_{\Lambda}p\left(\lambda\right)\prod_{x\in U}\mu_{R}^{O^{x}}(\rho'(U,x)(A)),
\label{Contextualidade}
\end{equation}
closing the representation of a $R$-empirical model as a classical system.

\begin{definition}
An $R$-empirical model is said $R$-non-contextual if there is a $R$-measure $p$ and a set of hidden variables $\Lambda$ such that Equation \ref{Contextualidade} holds for all $U\in\mathcal{U}$.
\end{definition}

Another property we can impose is called outcome-determinism, the property of logically distinguish between outcomes. In combination with non-disturbance, we get the following result \cite{Abramsky_2011,Abramsky_2017}:

\begin{proposition}
A non-disturbing $R$-empirical model that satisfies the outcome determinism condition has as its hidden variables exactly its global events.
\end{proposition}

This result allow to develop measures of contextuality by the use of linear program \cite{Abramsky_2017,barbosa2019continuousvariable}, once we know the set $\Lambda$. With it one can also prove the Fine–Abramsky–Brandenburger Theorem \cite{Abramsky_2011}, where $R$-contextuality can be understood as the nonextendability of a local section to a global section of $\mathcal{D}_{R}\mathcal{E}$, or in other words, as the nonexistence of a global $R$-measure with marginalization to a context $U\in\mathcal{U}$. We can graphically describe the factorization, now understood as an equivalent condition to non-contextual behavior, as the commutation of the diagram
\begin{equation}
    \begin{tikzcd}
\mathcal{E}(\mathcal{U}) \arrow{rr}{\mu_{R}} \arrow{dr}{i'} & & R \\
& \mathcal{E}(X) \arrow{ur}{\mu_{R}} &
\end{tikzcd}
\end{equation}
It is clear the rule of the global events, as defining a global $R$-measure $\mu_{R}^{O^{X}}$, and the commutation giving the realization of the $R$-empirical model by these global events. Here $i'$ is the inclusion of local events in global events.

\section{Čech cohomology}
\label{Čech}

To define the Čech cohomology, we need to use the covering $\mathcal{U}$, ignoring individual measurements. We define the nerve $\textsc{N}(\mathcal{U})$ as the collection of contexts and its non-empty intersections, with elements $\sigma=(U_{j_{0}},...,U_{j_{q}})$ given by $\left|\sigma\right|=\bigcap_{k=0}^{q}U_{j_{k}}\neq\emptyset$, with $j_{k}\in I$ ordered listed. The nerve of a covering defines an abstract simplicial complex, where the $q$-simplices forms a collection $\textsc{N}(\mathcal{U})^{q}$. See that for this to hold the contexts that form $\sigma$ must be different\footnote{There is an equivalence between ordered and unordered Čech cohomology, where in the last one we can repeat the contexts, but for clarity I will work with the ordered one.}. We can define the map 
\begin{equation}
    \partial_{j_{k}}:\textsc{N}(\mathcal{U})^{q}\to\textsc{N}(\mathcal{U})^{q-1}
\end{equation} 
as $\partial_{j_{k}}(\sigma)=(U_{j_{0}},...,\widehat{U_{j_{k}}},...,U_{j_{q}})$, where the hat means the omission of the context.

\subsubsection*{Cohomological contextuality}

The Čech cohomology is defined not by $\mathcal{D}_{R}\mathcal{E}$, but by a functor $\mathcal{F}:\textsc{N}(\mathcal{U})\to\mathbf{AbGrp}$ from the nerve to the category of Abelian groups, that represents the presheaf $\mathcal{D}_{R}\mathcal{E}$. Typically $\mathcal{F}=F_{S}\mathcal{L}$, where $\mathcal{L}$ at least a subsheaf of $\mathcal{E}$, with $F_{S}:\mathbf{Set}\to\mathbf{AbGrp}$ that assigns to a set $O$ the free Abelian group $F_{S}(O)$ generated by it related to the ring $S$. Specifically, we define it by a free module on a ring $S$, see Ref. \cite{abramsky_et_al:LIPIcs:2015:5416}. In other words, it is a representation of $\mathcal{E}$. Such presheaf must satisfy:
\begin{itemize}
    \item $\mathcal{F}(U)\neq\emptyset$ for all $U\in\mathcal{U}$
    \item $\mathcal{F}$ is flasque beneath the cover (the restriction map $\rho^{\mathcal{F}}(U',U)$ on $\mathcal{F}$ is surjective whenever $U\subseteq U'\subseteq V$, $V\in\mathcal{U}$
    \item any compatible family given by just one of the events for each context, here thought as an element of the basis, induces an unique global section.
\end{itemize}
The first condition implies that a context must have a non-trivial image by $F$, and the second one codifies that the information given by $F$ in a subcontext is already in the context that contains it. The third condition follows from the imposition of Gluing and Locality axioms on the basis and is induced by the sheaf structure of the events. 

We can define an augmented Čech cochain complex
\begin{equation}
    \begin{tikzcd}[row sep=tiny]
\mathbf{0} \arrow{r} & C^{0}(\mathcal{U},\mathcal{F}) \arrow{r}{d^{0}} & C^{1}(\mathcal{U},\mathcal{F}) \arrow{r}{d^{1}} & C^{2}(\mathcal{U},\mathcal{F}) \arrow{r}{d^{2}} & \dots
\end{tikzcd}
\end{equation}
where the Abelian group of $q$-cochains is
\begin{equation}
    C^{q}(\mathcal{U},\mathcal{F})=\prod_{\sigma\in\textsc{N}(\mathcal{U})^{q}}\mathcal{F}\left(\left|\sigma\right|\right)
\end{equation}
and the coboundary map $d^{q}:C^{q}(\mathcal{U},\mathcal{F})\to C^{q+1}(\mathcal{U},\mathcal{F})$ as the group homomorphism given by
\begin{equation}
    d^{q}(\omega)(\sigma)=\sum_{k=0}^{q+1}(-1)^{k}\rho'(\left|\partial_{j_{k}}\sigma\right|,\left|\sigma\right|)\omega(\partial_{j_{k}}\sigma)
\end{equation}
with $\omega\in C^{q}(\mathcal{U},\mathcal{F})$ and $\sigma\in\textsc{N}(\mathcal{U})^{q+1}$. One can show that $d^{q+1}d^{q}=0$, so we can construct the cohomology of this cochain complex.

The Abelian group of $q$-cocycles is defined by 
\begin{equation}
Z^{q}(\mathcal{U},\mathcal{F})=\left\{c\in C^{q}(\mathcal{U},\mathcal{F})|d^{q}c=0\right\}=\text{ker}(d^{q}).
\end{equation}
The Abelian group of $q$-coboundaries is defined by \begin{equation}
B^{q}(\mathcal{U},\mathcal{F})=\left\{c\in C^{q}(\mathcal{U},\mathcal{F})|c=d^{q-1}z,\text{ }z\in C^{q-1}(\mathcal{U},\mathcal{F})\right\}=\text{Im}(d^{q-1}).
\end{equation}
Clearly $B^{q}(\mathcal{U},\mathcal{F})\subseteq Z^{q}(\mathcal{U},\mathcal{F})$, and we define the $q$-th Čech cohomological group as the quotient $\Breve{H}^{q}(\mathcal{U},\mathcal{F})=Z^{q}(\mathcal{U},\mathcal{F})/B^{q}(\mathcal{U},\mathcal{F})$.

\subsubsection*{Standard results}

Let's see two results already known in the literature of Čech cohomology and cohomological contextuality \cite{Abramsky_2012}. The first result is the relation between compatible families and the $0$-th cohomological group. 

\begin{proposition}
There is a bijection between compatible families and elements of the zeroth cohomology group $\Breve{H}^{0}(\mathcal{U},\mathcal{F})$.
\end{proposition}

This result follows from the observation that on the augmented cochain complex, the coboundary group $B^{0}(\mathcal{U},\mathcal{F})$ is trivial. It allows the use of elements of $\Breve{H}^{0}(\mathcal{U},\mathcal{F})$ in the search for an extension to a compatible family.

The second result is the characterization of what is called cohomological contextuality. It consists of the construction of obstruction of an initial $0$-cochain that codifies the local sections of $\mathcal{F}$ on the covering $\mathcal{U}$. 

We will need some concepts. Let's define an auxiliary presheaf $\mathcal{F}|_{U}(V)=\mathcal{F}(U\cap V)$, and the canonical presheaf map $p:\mathcal{F}\to\mathcal{F}|_{U}::p_{V}:r\mapsto r|_{U\cap V}$. Another auxiliary presheaf is $\mathcal{F}_{\Bar{U}}(V)=\text{ker}(p_{V})$, defining the exact sequence of presheaves
\begin{equation}
    0\longrightarrow\mathcal{F}_{\Bar{U}}\longrightarrow\mathcal{F} \xrightarrow{\,\,\,\, p\,\,\,\,}\mathcal{F}|_{U}
\end{equation}
The $U$-relative cohomology is defined as the Čech cohomology of the presheaf $\mathcal{F}_{\Bar{U}}$.

One can show that the elements of $\Breve{H}^{0}(\mathcal{U},\mathcal{F}_{\Bar{U}_{j}})$ are in bijection with compatible families $\left\{r_{k}\right\}$ where $r_{j}=0$. Now, starting with a local section $s_{j_{0}}$ of $U_{j_{0}}$, the non-disturbance condition implies the existence of a family $\left\{s_{j_{k}}\right\}$ such that $s_{j_{0}}|_{j_{0}j_{k}}=s_{j_{k}}|_{j_{0}j_{k}}$ for all $k\neq 0$. Defining $c=\left\{s_{j_{k}}\right\}_{0\leq k}\in C^{0}(\mathcal{U},\mathcal{F})$, a trivial computation shows that $z=dc\in Z^{1}(\mathcal{U},\mathcal{F}_{\Bar{U}_{j_{0}}})$, and we define the obstruction $\gamma(s_{j_{0}})$ as the cohomology class $[z]\in\Breve{H}^{1}(\mathcal{U},\mathcal{F}_{\Bar{U}_{j_{0}}})$.

\begin{proposition}
Let $\mathcal{U}$ be connected, $U_{j_0}\in\mathcal{U}$, and $s_{j_0}\in\mathcal{F}(U_{j_0})$. Then $\gamma(s_{j_0})=0$ if and only if there is a compatible family $\left\{r_{j_k}\in\mathcal{F}(U_{j_{k}})\right\}_{U_{j_{k}}\in\mathcal{U}}$ such that $r_{j_0}=s_{j_0}$.
\end{proposition}

The idea behind this construction is to take a local section and ask if there is an extension to a global section, looking for the triviality of the obstruction. It is a sufficient condition for contextual behavior to have non-trivial obstruction for some local sections, but not a necessary one. A model that presents at least one non-trivial obstruction will be called logically cohomological contextual, to distinguish from the usual possibilistic contextual models.

\section{What is forgotten?}
\label{What}

The last condition imposed on $\mathcal{F}$ defines some limitations of what kind of element of $\mathcal{F}(U)$, defined as the free module on $O^{U}$, we have access to. To see that, if we choose $F_{S=\mathbb{Z}}$, we can define the element $\sum_{i}[s^i]$ of $\mathcal{F}$ adding all local sections $s_{i}$ of $\mathcal{E}(U)$ with non-null $R$-measure, $\mu_{\mathbb{B}}^{O^{U}}(s_{i})=1_{\mathbb{B}}$. By the non-disturbance condition, it has a compatible family, the model itself, but it can have possibilistic contextual behavior. Therefore, there is a restriction of which kind of local sections we can work in the cohomological framework. Usually, this is avoided by choosing just a single $s^{i}$ to start the analysis of the cohomological obstruction\footnote{Another way to see that is noting that it is equivalent to describe compatible local sections as the restriction of a global one. In the case of a linear combination of the local sections of a context, representing the possibilistic sections of a model, one could ask if the negative coefficients change this conclusion. If it was possible, we could write the incident matrix $M$ of the model and there will be a solution of the matrix equation $Mb=p$, see Ref. \cite{Abramsky_2017,Abramsky_2011}. Here $p$ is the possibilistic coefficients, the entries of the outcome table of the model that is in the Boolean semi-ring, and $b$ are the coefficients for the global sections. But the solution of this equation isn't granted even if we use integer coefficients in $b$ \cite{hermida1986,camion1971}.}.

\subsubsection*{Rings, semi-rings, and forgetful functors}

The intuition of the use of cohomology in contextuality is as a way to codify in an Abelian group the data in each context plus outcome set plus possible measures and use them to study its behavior. It is natural to use the free Abelian group generated by the outcome set given by a functor over the category of contexts, therefore codifying contexts and outcomes. We can understand it in the following way. Let's assume enough conditions as presented to contextual behavior appears as the non-factorizability of the model. Graphically 
\begin{equation}
    \begin{tikzcd}
\mathcal{E}(\textsc{N}(\mathcal{U})^{0}) \arrow{rr} \arrow{dr}{i'} & & S \\
& \mathcal{E}(X) \arrow{ur} &
\end{tikzcd}
\end{equation}
where we start with the data of $\mathcal{E}(\textsc{N}(\mathcal{U})^{0})\to S$, and ask for the existence of the data of $\mathcal{E}(X)\to S$ plus a map $i'$ such that the diagram commutes. We are asking for a factorization of the local events to global events when using $F_{S}$ as the tool to quantify the measurement scenario. These diagrams are the informal way to show the data in $\Breve{H}^{0}(\mathcal{U},\mathcal{F})\to C^{0}(\mathcal{U},\mathcal{F})$ for $\mathcal{F}=F_{S}\mathcal{E}$. The previous construction of cohomology ask for the data that can't be codified in a context, and cohomological contextuality is the existence of more data than the context can storage\footnote{The construction the previous section use a fact that was explicitly showed later, that any non-contextual empirical model admits a morphism with the trivial empirical model \cite{Karvonen_2019}. The obstruction can then be understood as the identification of a failure in the morphisms of the category of empirical models, one that forbids the morphism between the trivial empirical model and the model in question.}.

But the choice of the ring $S$ for the definition of the group does not come from the measure. This is the reason for the impossibility of the characterization with standard Čech cohomology\footnote{The use of $S=\mathbb{Z}$ is usual choice. The justification in contextuality is the use in possibilistic models, the ones with $\mathbb{B}$-measures, and strong contextual models \cite{Abramsky_2017}. But as shown in Ref. \cite{Car__2017}, even for strong contextual models the characterization by cohomology can fail.}.

The categories of rings, Abelian groups, and modules are canonically related, as the categories of semi-rings, Abelian monoids, and semi-modules. So the relation between these algebraic structures is simply codified by the canonical forgetful functor $F:\textbf{Ring}\to\textbf{Rig}$, that forget the negatives of the ring category. To study such functor, let's revise what we can forget with it \cite{Baez_2009}. 

\begin{definition}
Let $F:\mathcal{A}\to\mathcal{B}$ be a functor, it
\begin{itemize}
    \item is called full if for all pair of objects $A,A'$ of $\mathcal{A}$ the map $F:\text{hom}(A,A')\to\text{hom}(F(A),F(A'))$ induced by $F$ is surjective; it forgets structure if it is not full;
    \item is called faithful if the induced map is injective for all pairs of objects; it forgets stuff if it is not faithful;
    \item is called essentially surjective if for any object $B$ of $\mathcal{B}$ there is an object $A$ of $\mathcal{A}$ such that $F(A)$ is isomorphic to $B$; it forgets properties if it is not essentially surjective;
    \item forgets nothing if it is full, faithful, and essentially surjective, or equivalently $F$ admits an inverse, and the categories are called equivalents.
\end{itemize}
\end{definition}

A known result is that a semi-ring can be completed into a ring if and only if the semi-ring is cancellative\footnote{The left adjoint of the forgetful functor $F:\textbf{Ring}\to\textbf{Rig}$ is not monic if the additive monoid of $R$ is not cancellative.}. As an example, the Boolean ring $\mathbb{B}$ isn't cancellative:
\begin{equation}
    1_{\mathbb{B}}+1_{\mathbb{B}}=1_{\mathbb{B}}=1_{\mathbb{B}}+0_{\mathbb{B}}
\end{equation}
but if the cancellative holds, we conclude that $1_{\mathbb{B}}=0_{\mathbb{B}}$, absurd. Therefore the functor $F:\textbf{Ring}\to\textbf{Rig}$ forgets property. Also, there are more homomorphisms between objects of $\textbf{Rig}$ than in objects of $\textbf{Ring}$, implying that $F$ forgets structure. It doesn't forget stuff once every ring in $\textbf{Ring}$ is a semi-ring in $\textbf{Rig}$, with the morphisms being preserved. 
In conclusion, there are too much structure and property in $\textbf{Ring}$, the category we use to define the cohomology with Abelian groups. Extra property plus structure allows the violation of the cohomological characterization of contextual behavior, there are too many ways to justify non-contextuality\footnote{Categorically, we can construct a natural transformation between the functors $\mathcal{D}_{R}$ and $F_{S}$. With the forgetful map, one can show that usually there isn't natural equivalence between them, and this allows the violations.}. Explicitly, the existence of negatives allows the $"Z"$ sections, causing the known violations, as shown in the following example.

\subsubsection*{Examples of violation}

The possibilistic empirical model given by Table \ref{Table} and its bundle representation in Figure \ref{example} is an explicit example of the violation of the cohomological contextuality. It's codified in the table
\begin{equation}
\label{Table}
\begin{tabular}{lllll}
\hline\noalign{\smallskip}
& \large{$00$} & \large{$01$} & \large{$10$} & \large{$11$} \\
\noalign{\smallskip}\hline\noalign{\smallskip}
\large{$ab$} & \large{$1_{\mathbb{B}}$} & \large{$0$} & \large{$0$} & \large{$1_{\mathbb{B}}$} \\
\noalign{\smallskip}\hline\noalign{\smallskip}
\large{$bc$} & \large{$1_{\mathbb{B}}$} & \large{$0$} & \large{$0$} & \large{$1_{\mathbb{B}}$} \\
\noalign{\smallskip}\hline\noalign{\smallskip}
\large{$bc$} & \large{$1_{\mathbb{B}}$} & \large{$0$} & \large{$0$} & \large{$1_{\mathbb{B}}$} \\
\noalign{\smallskip}\hline\noalign{\smallskip}
\large{$da$} & \large{$1_{\mathbb{B}}$} & \large{$1_{\mathbb{B}}$} & \large{$1_{\mathbb{B}}$} & \large{$1_{\mathbb{B}}$} \\
\noalign{\smallskip}\hline
\end{tabular}
\end{equation}
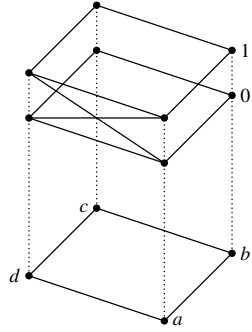
\begin{figure}
\centering
\scalebox{0.3}{
\begin{tikzpicture}

\draw [ultra thick] (-9,12) -- (-3,10) node[right] {\Huge{$\ a$}};
\draw [ultra thick] (-3,10) -- (0,13) node[right] {\Huge{$\ b$}}; 
\draw [ultra thick] (0,13) -- (-6,15)
node[left] {\Huge{$c\ $}}; 
\draw [ultra thick] (-6,15) -- (-9,12) node[left] {\Huge{$d\ $}};

\filldraw [black] (-9,12) circle (4pt);
\filldraw [black] (-3,10) circle (4pt);
\filldraw [black] (0,13) circle (4pt);
\filldraw [black] (-6,15) circle (4pt);

\draw[loosely dotted, ultra thick] (-9,12) -- (-9,21);
\draw[loosely dotted, ultra thick] (-3,10) -- (-3,19);
\draw[loosely dotted, ultra thick] (0,13) -- (0,22);
\draw[loosely dotted, ultra thick] (-6,15) -- (-6,24);

\filldraw [black] (-9,21) circle (4pt);
\filldraw [black] (-3,19) circle (4pt);
\filldraw [black] (0,22) circle (4pt);
\filldraw [black] (-6,24) circle (4pt);

\filldraw [black] (-9,19) circle (4pt);
\filldraw [black] (-3,17) circle (4pt);
\filldraw [black] (0,20) circle (4pt);
\filldraw [black] (-6,22) circle (4pt);

\draw [ultra thick] (-9,19) -- (-3,19);
\draw [ultra thick] (-3,19) -- (0,22) node[right] {\Huge{$\ 1$}}; 
\draw [ultra thick] (0,20) -- (-6,22); 

\draw [ultra thick] (-9,21) -- (-3,17);
\draw [ultra thick] (-3,17) -- (0,20) node[right] {\Huge{$\ 0$}}; 
\draw [ultra thick] (0,22) -- (-6,24); 

\draw [ultra thick] (-9,21) -- (-3,19);
\draw [ultra thick] (-9,19) -- (-3,17);

\draw [ultra thick] (-6,22) -- (-9,19);
\draw [ultra thick] (-6,24) -- (-9,21);


\end{tikzpicture}}
\caption{Bundle diagram of the $\mathbb{B}$-empirical model of the Table \ref{Table}.}
\label{example}
\end{figure}

First, this model shows that we need to choose which section we will start with. If we choose to extend the section
\begin{equation}
    s=[(da)\to(00)]+[(da)\to(01)]+[(da)\to(10)]+[(da)\to(11)]
\end{equation}
it has only one possible compatible family to describe it, the rest of the model. But as a contextual model, there isn't a global section \footnote{Even if we allow negative coefficients. This model is the combination of the trivial one with Hardy's model, and there isn't a Boolean ring solution for this model.}.

Let's choose one local section, and not a linear combination of them. The cohomology will be given by the free Abelian group $F_{\mathbb{Z}}$ as already discussed, with the integers as coefficients. One can easily show that all of them admit an extension with trivial obstruction because we can use the $"Z"$ structure on the context $da$. For horizontal sections it's trivial, and for diagonal sections we can use the other diagonal one and the negative of a horizontal section in the same context to achieve a global section without obstruction, for example
\begin{equation}
\begin{split}
    s=(&[(da)\to(01)]-[(da)\to(00)]+[(da)\to(10)],\\
    &[(ab)\to(11)],[(bc)\to(11)],[(cd)\to(11)]).
\end{split}
\end{equation}
The first term is the $"Z"$ structure for the extension of any of the diagonal sections. It's an example of a contextual model that shows logically cohomological non-contextuality (for more examples, even for strong contextual models, see Ref. \cite{caru2018towards,Car__2017}).

Another known example of such a violation is interesting. Usually, contextuality is explored as a probabilistic behavior, with the positive real numbers as the semi-ring $R$ of the empirical model. The natural choice is a functor $F_{\mathbb{R}}$ where the free Abelian group has coefficients in the ring of reals and works with the cohomology of it. We start with the local sections and codify the probabilities in the coefficients. Even with $R\in\mathbb{R}$, a result in Ref. \cite{Abramsky_2011} for finite outcome sets shows the equivalence between non-disturbing probabilistic empirical models and probabilistic empirical models that can be realized with negative real numbers. Therefore $F_{\mathbb{R}}$ shows trivial obstruction for non-disturbing models, an extreme example of violation in the characterization of contextual behavior with standard cohomology.

\section{Čech cohomology on semi-modules}
\label{semi-modules}

Following the points raised in the discussion of the previous section, let's construct the cohomological tools presented but for semi-rings. The idea is to use the coefficients to codify the measures, naturally restricting all the construction to semi-rings. The forgetting of property imposes that any enough general use of cohomological tool must use only the semi-ring and not a ring that completes it.

\subsubsection*{Cohomology on semi-modules}

The cohomology with semi-modules on a semi-ring $R$ has as a difference the impossibility of defining the coboundary operators \cite{patchkoria2006exactness}. Such maps explicitly depend on subtraction, something that usually doesn't exist in the realm of semi-rings.

\begin{definition}
A cochain complex of $R$-semi-modules $C=\left\{C^{q},d^{q}_{+},d^{q}_{-}\right\}_{q\in\mathbb{Z}}$ consists of $R$-semi-modules $C^{q}$ and $R$-homomorphisms $d^{q}_{+},d^{q}_{-}$ as:
\begin{equation}
    C=\dots\xbigtoto[d^{q-2}_{-}]{d^{q-2}_{+}}C^{q-2}\xbigtoto[d^{q-1}_{-}]{d^{q-1}_{+}}C^{q}\xbigtoto[d^{q}_{-}]{d^{q}_{+}}C^{q+1}\xbigtoto[d^{q+1}_{-}]{d^{q+1}_{+}}\dots
\end{equation}
that satisfy the condition
\begin{equation}
    d^{q+1}_{+}\circ d^{q}_{+}+d^{q+1}_{-}\circ d^{q}_{-}=d^{q+1}_{-}\circ d^{q}_{+}+d^{q+1}_{+}\circ d^{q}_{-}.
\end{equation}
\end{definition}

We define the $R$-semi-module of $q$-cocycles as $Z^{q}=\left\{c\in C^{q}|d^{q}_{+}(c)=d^{q}_{-}(c)\right\}$, and the $q$-th cohomology $R$-semi-module as $H^{q}(C)=Z^{q}(C)/\rho^{q}$, with $\rho^{q}$ a congruence relation in $Z^{q}(C)$. There are different relations in the literature that define different cohomologies, for example $x\rho^{q}y$ can means that for some $u,v\in C^{q-1}$ holds
\begin{equation}
    x+d^{q-1}_{+}(u)+d^{q-1}_{-}(v)=y+d^{q-1}_{+}(v)+d^{q-1}_{-}(u)
\end{equation}
or that for some $u\in C^{q-1}$ holds
\begin{equation}
    x+d^{q-1}_{+}(u)=y+d^{q-1}_{-}(u)
\end{equation}
(see Ref. \cite{Patchkoria2014CohomologyMO,Patchkoria2017CohomologyMO} for details). Usually, the choice of one relation depends on how difficult is to deal with these algebraic objects.

A cochain sequence of modules $G=\left\{G^{q},d^{q}_{+},d^{q}_{-}\right\}$ is a cochain complex in the above sense of groups if and only if $G'=\left\{G^{q},d^{q}_{+}-d^{q}_{-}\right\}$ is the cochain complex of modules. This definition is a direct generalization of the usual cohomology on modules, and the cohomological semi-modules of $G$ are the modules of the cohomology of $G'$ in the usual sense.

\subsubsection*{Čech cohomology with semi-rings}

The Čech cohomology on $R$-semi-modules is defined as usual, but one needs to work with the differentials separately \cite{Jun_2017}. We have a covering $\mathcal{U}$, and we can construct its nerve $\textsc{N}(\mathcal{U})$ and the map $\partial_{j_k}$. We substitute the presheaf $\mathcal{F}$ of Abelian groups by a presheaf of $R$-semi-modules $\mathcal{G}$. 

The cochain complex is $C^{q}(\mathcal{U},\mathcal{G})=\prod_{\sigma\in\textsc{N}(\mathcal{U})^{q}}\mathcal{G}(\left|\sigma\right|)$, and the coboundary operators are defined by
\begin{equation}
    d^{q}_{+}(\omega)(\sigma)=\sum_{k=0, even}^{q+1}\rho'(\left|\partial_{j_{k}}\sigma\right|,\left|\sigma\right|)\omega(\partial_{j_{k}}\sigma)
\end{equation}
\begin{equation}
    d^{q}_{-}(\omega)(\sigma)=\sum_{j=0, odd}^{q+1}\rho'(\left|\partial_{j_{k}}\sigma\right|,\left|\sigma\right|)\omega(\partial_{j_{k}}\sigma).
\end{equation}
We define the $R$-semi-module of $q$-cocycles $Z^{q}(\mathcal{U},\mathcal{G})$ and the $q$-th Čech cohomology $R$-semi-module $\Breve{H}^{q}(\mathcal{U},\mathcal{G})=Z^{q}(\mathcal{U},\mathcal{G})/\rho^{q}$, for a chosen relation $\rho^{q}$.

\subsubsection*{Generalization of first result}

The notion of compatible family is independent of the algebraic structure, but its relation with  $\Breve{H}^{0}(\mathcal{U},\mathcal{G})$ must be re-explored.

\begin{proposition}
There is a bijection between compatible families and elements of the zeroth cohomology $R$-semi-module $\Breve{H}^{0}(\mathcal{U},\mathcal{G})$.
\end{proposition}

\begin{proof}
By definition $\Breve{H}^{0}(\mathcal{U},\mathcal{G})=Z^{0}(\mathcal{U},\mathcal{G})/\rho^{0}$. But the relation depends on the choice: $x\rho^{0}y\iff x+d^{-1}_{+}(u)+d^{-1}_{-}(v)=y+d^{-1}_{+}(v)+d^{-1}_{-}(u)$ for some $u,v\in C^{-1}(\mathcal{U},\mathcal{G})$, follows from the definition $C^{-1}(\mathcal{U},\mathcal{G})=0$ that $x=y$, and therefore $\Breve{H}^{0}(\mathcal{U},\mathcal{G})=Z^{0}(\mathcal{U},\mathcal{G})$. Usually this result is true for any relation that generalize the condition of coboundaries. As \begin{equation}
Z^{0}(\mathcal{U},\mathcal{G})=\left\{y=(y_{U_{i}})\in C^{0}(\mathcal{U},\mathcal{G})=\prod_{i\in I}\mathcal{G}(U_{i})|d^{0}_{+}(y)=d^{0}_{-}(y)\right\}    
\end{equation}
and $d^{0}_{+}$ ($d^{0}_{-}$) is a product of maps $\mathcal{G}(U_{j})\to\mathcal{G}(U_{i}\cap U_{j})$ ($\mathcal{G}(U_{i})\to\mathcal{G}(U_{i}\cap U_{j})$, $j>i$) induced by the inclusion of contexts $U_{i}\cap U_{j}\to U_{j}$ ($U_{i}\cap U_{j}\to U_{i}$), we get $y_{U_{i}}|_{U_{i}\cap U_{j}}=y_{U_{j}}|_{U_{i}\cap U_{j}}$.
\end{proof}

\section{Obstruction}
\label{Obstruction}

We will impose that $\mathcal{G}$ satisfies the same three conditions of $\mathcal{F}$. We define it as a free $R$-semi-module generated by the local events for each element of $\mathcal{U}$ and intersections. Again, we need to restrict the section that we can explore.

The loss of the structures and properties has consequences. One of them is the impossibility of following the same path to relate the obstruction with contextuality because the notion of obstruction depends on the subtraction of the differentials. Without it, the usual definition of the differential $z$ as $d^{q}_{+}(c)=d^{q}_{-}(c)+z$ is not well defined. Therefore, there may not be a $z\in C^{1}(\mathcal{U},\mathcal{G})$ to work. An example is the probabilistic semi-ring: there isn't a probability measure $z$ such that $a=b+z$.

The beautiful relation between the cocycles and coboundaries is also lost. The quotient necessary to define the cohomology is given by a relation, and not by a semi-module or any other closed algebraic structure. There is another problem: which relation one chooses to define the cohomology? All the generalizations to semi-modules use different relations that become equivalent for the usual notion of coboundary for a module, but they are non-equivalents in their more general form.

\subsubsection*{Difference operator}

We need a formalism that allows the construction of obstruction without subtraction. It is a function, and its uniqueness to define the difference cocycle is what justifies its use. The difference between the differentials is codified by $f(d^{q}_{+}c)=d^{q}_{-}c=d^{q}_{+}c-z$, with a unique $z$. Without such a function, we need to search for another function with the ability to capture the difference between two distinct marginalizations in the same intersection. 

It exists as a cochain of operators in the $R$-semi-modules if we allow $R$ to be a semi-field. For the probabilistic case, for example, it would be a cochain of stochastic operators, and we need the product inversion to satisfy the probability condition. In analogy with stochastic operators, we can make the following definition.

\begin{definition}
An operator is called row $R$-stochastic if the sum of elements in every row is $1_{R}$, and column $R$-stochastic if its transpose is row $R$-stochastic.
\end{definition}

Let $g[\sigma]d^{q}_{+}(c)(\sigma)=d^{q}_{-}(c)(\sigma)$ defines the cochain of operators, such that for each element of $\sigma\in\text{N}(\mathcal{U})^{q+1}$ we have an operator. We can define for each $c\in C^{q+1}(\mathcal{U},\mathcal{G})$ an cochain of operators $g$ such that $gd^{q}_{+}(c)=d^{q}_{-}(c)$. The pair of cochain $c$ and one of its cochain of operators $g$ will be called difference cochain, and denoted by $(g,c)$. 

Usually there isn't an unique cochain of operators $g$ that satisfies it, so we define the equivalence relation $[g_{q},c]$ of cochains given by $[h,c]\sim [g,c]$ if $gd^{q}_{+}(c)=hd^{q}_{+}(c)$, and by abuse we will called it the difference cochain of $c$. The differentials couldn't define an element $z=dc$ of $c$ without negatives but can define a class of difference cochains $[g_{q},c]$ of $c$ if we restrict to semi-fields. As the relation between cochains is broken, $[g_{q},c]$ usually doesn't have a representation on the next cochain of the cochain complex.

Given a cochain $c\in C^{q}(\mathcal{U},\mathcal{G})$, a row $R$-stochastic operator exists to represent each entry of the difference cochain $[g_{q},c]$, uniquely defined. If $[Id,c]=[g_{q},c]$, then $d^{q}_{+}c=d^{q}_{-}c$, therefore $c$ is a cocycle. For what follows, the notion of coboundary is unnecessary. 

\begin{definition}
The difference operator is a function $[g_{q}]::c\in C^{q}(\mathcal{U},\mathcal{G})\mapsto [g_{q},c]$, defining for each cochain its unique class of difference cochains.
\end{definition}

Formally, we can thing the following diagram
\begin{equation}
    \begin{tikzcd}[row sep=tiny]
& & C^{1}(\mathcal{U},\mathcal{G}) \arrow{dd}{[g_{0}]} \arrow{dr}{id} & & C^{2}(\mathcal{U},\mathcal{G}) \arrow{dd}{[g_{1}]} \arrow{dr}{id} & &\\
0 \arrow[r] & C^{0}(\mathcal{U},\mathcal{G}) \arrow{ur}{d^{0}_{+}} \arrow{dr}{d^{0}_{-}} & & C^{1}(\mathcal{U},\mathcal{G}) \arrow{ur}{d^{0}_{+}} \arrow{dr}{d^{0}_{-}} && C^{2}(\mathcal{U},\mathcal{G}) & \dots\\
& & C^{1}(\mathcal{U},\mathcal{G}) \arrow{ur}{id} & & C^{2}(\mathcal{U},\mathcal{G}) \arrow{ur}{id} & & &
\end{tikzcd}
\end{equation}
where the operator $[g_{q}]$ is explicit. One thing that this diagram shows is the lost of the structure between each $q$: one can't takes anymore that $d^{q+1}d^{q}=0$ has a nice generalization, therefore a initial cochain can present non-trivial difference functions $g_{q}$ for each $q$\footnote{This fact result in the lost of exactness in the relative cohomology used in Ref. \cite{Car__2017,abramsky_et_al:LIPIcs:2015:5416}, and this is the justification to build a generalization of the construction in Ref. \cite{Abramsky_2012}.}. 

\subsubsection*{Generalized obstruction}

Lets use $\mathcal{G}$, and keep our models non-disturbing as imposed by the flasque beneath the cover condition. The restriction due to the impossibility of negatives restrict us to explore the sections $\mu_{R}^{O^{U}}(x)[x]$ of $\mathcal{G}$ with $x\in O^{U}$. Again, we want to describe is the extension of a section of $\mathcal{G}$, such that the $R$-measure that starts in it also finish in it through the model. Using again the exact sequence of presheaves
\begin{equation}
    0\longrightarrow\mathcal{G}_{\Bar{U}}\longrightarrow\mathcal{G} \xrightarrow{\,\,\,\, p\,\,\,\,}\mathcal{G}|_{U}
\end{equation}
the projection $p$ acts on $[g,c]=[g_{0},c]$ canonically as $p_{U}:C^{q}(\mathcal{U},\mathcal{G})\to C^{q}(\mathcal{U},\mathcal{G}|_{U})::[g,c]\mapsto [g|_{U},p_{U}(c)]$, the action on the projected difference cochain. Such restriction usually doesn't preserve the equivalence relation, since we could have a non-trivial $[g,c]$ such that $[g|_{U},p_{U}(c)]=[Id|_{U},p_{U}(c)]$, and we will use it to define the obstruction. The kernel of $p_{U}$ is the set of difference cochains $[g,c]$ that satisfy $p_{U}[g,c]=[Id|_{U},p_{U}(c)]$\footnote{Without subtraction, one can't reach $0_{R}$ so easily. The difference operator generalize it, but its trivialization is represented by the trivial difference cochain $[Id,c]$. Such generalization is analogous to the impossibility of working with the algebra of the tangent space of a variety, but only directly with the variety elements. If we are working with a module cochain, we can represent the difference cochain as an element $z$ such that $p_{U}dz=0$, impling the usual sequence of presheafs.}, therefore a difference cochain $[g,c]$ of $C^{0}(\mathcal{U},\mathcal{G})$ is a difference cochain of $C^{0}(\mathcal{U},\mathcal{G}_{\Bar{U}})$ if $[g|_{U},p_{U}(c)]=[Id|_{U},p_{U}(c)]$.

Let $c_{j_{0}}$ be the initial section, one that is allowed by the presheaf $\mathcal{G}$. As before, non-disturbance implies the existence of a cochain $c=\left\{c_{j_{k}}\right\}$ such that $c_{j_{0}}|_{j_{0}j_{k}}=c_{j_{k}}|_{j_{0}j_{k}}$ for all $k$, and it defines a unique difference cochain $[g,c]$.

\begin{proposition}
A difference cochain $[g,c]$ as defined above is a difference cochain in $C^{1}(\mathcal{U},\mathcal{G}_{\Bar{U}_{j_{0}}})$.
\end{proposition}

\begin{proof}
Applying $p$ we get $[g|_{j_{0}},p_{j_{0}}(c)]=[Id|_{j_{0}},p_{j_{0}}(c)]$, since for $j_{k}>j_{l}$ we have
\begin{equation}
    p_{j_{0}}\left(d_{+}^{0}c(j_{k})(j_{l})\right)=c_{j_{k}}|_{j_{0}j_{k}j_{l}}=\left(c_{j_{0}}|_{j_{0}j_{k}}\right)|_{j_{l}}=c_{j_{0}}|_{j_{0}j_{k}j_{l}}
\end{equation}
and
\begin{equation}
    p_{j_{0}}\left(d_{-}^{0}c(j_{k})(j_{l})\right)=c_{j_{l}}|_{j_{0}j_{k}j_{l}}=\left(c_{j_{0}}|_{j_{0}j_{l}}\right)|_{j_{k}}=c_{j_{0}}|_{j_{0}j_{k}j_{l}}
\end{equation}
implying that $p_{j_{0}}\left(d_{+}^{0}c(j_{k})(j_{l})\right)=p_{j_{0}}\left(d_{-}^{0}c(j_{k})(j_{l})\right)=p_{j_{0}}\left(g\left(d_{+}^{0}c(j_{k})(j_{l})\right)\right)$, for all $j_{k}$ and $j_{l}$. By exactness we have that $[g,c]$ is a difference cochain in $C^{1}(\mathcal{U},\mathcal{G}_{\Bar{U}_{j_{0}}})$, with $[g,c]=[g_{\Bar{U}_{j_{0}}},c]$.
\end{proof}

The difference cochain $\gamma(c_{j_{0}})=[g_{\Bar{U}_{j_{0}}},c]$ is defined as the obstruction of $c_{j_{0}}$, doing the analogue role of $\gamma(c_{j_{0}})=[z]=[dc]$ in $Z^{1}(\mathcal{U},\mathcal{F}_{\Bar{U}_{j_{0}}})$, that can't be reach in general. We can now prove the main result of the paper.

\begin{theorem}
Let $\mathcal{U}$ be connected, $U_{j_{0}}\in\mathcal{U}$, and $c_{j_{0}}\in\mathcal{G}(U_{j_{0}})$. Then $\gamma(c_{j_{0}})$ is trivial if and only if there is a compatible family $\left\{r_{j_{k}}\in\mathcal{G}(U_{j_{k}})\right\}_{U_{j_{k}}\in\mathcal{U}}$ such that $c_{j_0}=r_{j_0}$.
\end{theorem}

\begin{proof}
The proof follows the diagram
\begin{equation}
    \begin{tikzcd}[row sep=small]
C^{0}(\mathcal{U},\mathcal{G}|_{U_{j_0}}) \arrow{rr}{d^{0}_{+}|_{U_{j_0}}} \arrow{dr}{d^{0}_{-}|_{U_{j_0}}} & & C^{1}(\mathcal{U},\mathcal{G}|_{U_{j_0}}) \arrow{dl}{[g|_{U_{j_{0}}}]}\\
& C^{1}(\mathcal{U},\mathcal{G}|_{U_{j_0}}) & \\
C^{0}(\mathcal{U},\mathcal{G}) \arrow{uu} \arrow{rr}[near end]{d^{0}_{+}} \arrow{dr}{d^{0}_{-}} & & C^{1}(\mathcal{U},\mathcal{G}) \arrow{uu} \arrow{dl}{[g]}\\
& C^{1}(\mathcal{U},\mathcal{G}) \arrow{uu} & \\
C^{0}(\mathcal{U},\mathcal{G}_{\Bar{U}_{j_0}}) \arrow{uu} \arrow{rr}[near end]{d^{0}_{+\Bar{U}_{j_0}}} \arrow{dr}{d^{0}_{-\Bar{U}_{j_0}}} & & C^{1}(\mathcal{U},\mathcal{G}_{\Bar{U}_{j_0}}) \arrow{uu} \arrow{dl}{[g_{\Bar{U}_{j_0}}]}\\
& C^{1}(\mathcal{U},\mathcal{G}_{\Bar{U}_{j_0}}) \arrow{uu} & 
\end{tikzcd}
\end{equation}
Suppose that $\gamma(c_{j_{0}})$ is trivial. Then there is a cochain $c\in C^{0}(\mathcal{U},\mathcal{G}_{\Bar{U}_{j_{0}}})$ such that $[g_{\Bar{U}_{j_{0}}},c]=[Id_{\Bar{U}_{j_{0}}},c]$. As $[g_{\Bar{U}_{j_{0}}},c]=[g,c]$, and $[g|_{U_{j_{0}}},p_{U_{j_{0}}}(c)]=[Id|_{U_{j_{0}}},c]$, by exactness we have $[g,c]=[Id,c]$. This last relation implies that $d^{0}_{+}c=d^{0}_{-}c$, i.e., $c=\left\{r_{j_{k}}\in\mathcal{G}(U_{j_{k}})\right\}_{U_{j_{k}}\in\mathcal{U}}$ is a compatible family, and as $c$ is constructed from $c_{j_{0}}$, we have $r_{j_{0}}=c_{j_{0}}$. In the diagram, if the bottom of the diagram is trivial, then the middle is trivial, since the top is trivial as well.

Now, suppose there is a compatible family $c=\left\{r_{j_{k}}\in\mathcal{G}(U_{j_{k}})\right\}_{U_{j_{k}}\in\mathcal{U}}$ with $r_{j_{0}}=c_{j_{0}}$. By the triviality of $[g,c]$ that follows from the compatibility of the family, and the exactness of the sequence of presheaves, we get $[g_{\Bar{U}_{j_{0}}},c]=[Id_{\Bar{U}_{j_{0}}},c]$, and therefore we have $\gamma(c_{j_{0}})=[g_{\Bar{U}_{j_{0}}},c]$ trivial. In the diagram, if the middle of the diagram is trivial, than the bottom is trivial.
\end{proof}

If we suppose that the semi-ring $R$ is cancellative, thus admitting a representation as a positive cone of a ring $S$, we can recover the usual use of cohomology to characterize contextual behavior once liberating the use of negatives. First, the relation between coboundary maps is recovered in a unique $S$-module of coboundaries, bringing all the formalism presented before. In both we start with a cochain $c$ codifying the $R$-measure, and identify
\begin{equation}
    z=dc=d^{0}_{+}c-d^{0}_{-}c
\end{equation}
where $z$ is the representation of the difference cochain $(g_{0},c)$ in $C^{1}(\mathcal{U},\mathcal{F})$, related by 
\begin{equation}
    z=d^{0}_{+}c-d^{0}_{-}c=(Id-g_{0})d^{0}_{+}c.
\end{equation}
While a cocycle as $z\in Z^{1}(\mathcal{U},\mathcal{F})$ is unique up to sum with a coboundary, thus defining a cohomological class, the difference cochain also presents uniqueness up to an equivalence class of operators. Finally, the obstruction $[z]$ is only trivial if and only if the generalized obstruction is trivial, once
\begin{equation}
        [z]=[d^{0}_{+}c-d^{0}_{-}c]=[(Id-g_{0})]d^{0}_{+}c
\end{equation}
holds, implying in $[z]=0$ if and only if $[(Id-g_{0})]=0$, or equivalently $[Id]=[g_{0}]$. The equivalence of the theorems if we allow the use of negatives follows directly of these identifications. 

\subsubsection*{$R$-contextuality}

The trivialization of all the obstructions of a model implies the extendability of all the local sections. In this sense, it is the verification of a paradox: the non-existence of an extension to a global section of a local section implies non-trivial obstruction of this section; on the other hand, non-trivial obstructions result in a paradox, therefore $R$-contextual behavior.

For this holds, $\mathcal{G}$ must preserve the data of the $R$-empirical model as a family of local sections of $\mathcal{D}_{R}\mathcal{E}$. It is the idea of $\mathcal{G}$ so far: the $R$-measures of each context of a model define an element of $C^{0}(\mathcal{U},\mathcal{G})$ uniquely, as a section of $\mathcal{G}$, with the free semi-module defined by the outcome sets and the semi-field $R$. The model is translated as a cochain of vectors of a $R$-semi-module that satisfies the normalization condition, with a basis of local events and coefficients as the $R$-measure of each of them. We have limited access to sections of $\mathcal{G}$, so we work just with the projected version of this cochain on the basis of the free semi-module.

\begin{corollary}
A model is $R$-contextual if and only if there is a local section $s_{i}$ representing a local event by $\mathcal{G}$ with non-trivial obstruction.
\end{corollary}

\begin{proof}
If the model is $R$-non-contextual, then any section admits extension to a global section, i.e., it is feasible as a linear combination of $R$ weighted global sections of $\mathcal{E}(X)$. Therefore all accessible obstructions are trivial. The negative of this reasoning implies that if any obstruction is non-trivial, then we have $R$-contextual behavior.

If it is contextual, then it is not feasible as a linear combination of $R$ weighted global sections of $\mathcal{E}(X)$. In other words, there is a local section that cannot be written as marginalization in its context of a combination of global events, in the sense of weights in $R$ forming a hidden-variable model. So there will be a non-trivial obstruction for this event.
\end{proof}

This result also generalizes the use of groups to explore contextual models \cite{Okay_2020,Sidiney_2021}. The failures of a formalism resulted from the imposition of a group structure are not present in this generalized cohomological approach.

\section{Framework of effect algebras}
\label{Effect}

A similar result already exists in the level of effect algebras \cite{Roumen_2017}. I will present the main ideas and compare the two results.

\begin{definition}
An effect algebra is a set $A$ with two operations $\vee:A\times A\to A::(a,b)\mapsto a\vee b$ and $\bot:A\to A::a\mapsto a^{\bot}$, and elements $0,1\in A$ such that
\begin{itemize}
    \item if $a\vee b$ is defined, then so $b\vee a$ and $a\vee b=b\vee a$;
    \item if $a\vee b$ and $(a\vee b)\vee c$ are defined, then so $b\vee c$ and $a\vee (b\vee c)$, and $(a\vee b)\vee c=a\vee (b\vee c)$;
    \item $0\vee a$ is defined and $a=0\vee a$;
    \item for all $a\in A$, $a^{\bot}$ is the unique element that satisfies $a^{\bot}\vee a=1$;
    \item if $a\vee 1$ is defined, then $a=0$.
    \item 
\end{itemize}
\end{definition}

\subsubsection*{Equivalence}

The first point is to turn explicit the relationship between the image of the functor of events $\mathcal{E}$ of the scenario, and the algebra of effects. This relation is given by Ref. \cite{10.1007/978-3-662-47666-6_32}, where the authors construct a faithful and full functor between the category of effect algebras and the category of measurement scenarios via objects called tests.

\begin{definition}
A test of a effect algebra $A$ is a finite set $\{a_{i}\in A\}_{i\in I}$, such that $\bigvee_{i}a_{i}$ is defined and $\bigvee_{i}a_{i}=1$. It is called an $n$-test if $|I|=n$.
\end{definition}

This is exactly the definition of measurement in a measurement scenario. Such functor can be seen by the use of contexts. A context defines a local measurement for all its local events, thus a finite $\sigma$-algebra, and it defines a Boolean effect algebra. It is known that effect algebras are naturally non-disturbing, see Ref. \cite{Wester_2018}, which implies that we get an effect algebra that can be decomposed as a union of Boolean algebras from a non-disturbing measurement scenario. The intuition follows the identification of the sets of local events on each context as a Boolean algebra, thus an effect algebra, and the non-disturbance condition glues them in a measurement scenario that defines an effect algebra. In the other way, we can define the tests of an effect algebra, and it already defines a measurement scenario. The contexts can be defined by the notion of simulability \cite{Guerini_2017}, the existence of a measurement to describe all the measurements of the context, therefore the identification of Boolean effect sub-algebras.

\subsubsection*{Cyclic cohomology}

I will give just an informal description of the cohomologies used to characterize contextual behavior, focusing on analogous points in the construction with the previous framework. For a deep and formal description, including the standard tool of cohomology theory in this framework, see Ref. \cite{Roumen_2017}.

For effect algebras, the cyclic cohomology $HC$ is defined as the relation between different possible tests. The first element of the complex is the set of $2$-tests, which by definition of test is isomorphic to the set of effects. The first group is
\begin{equation}
    HC^{1}(A)\cong\{\alpha:A\to\mathbb{R};\alpha(a\vee b)=\alpha(a)+\alpha(b), \alpha(a^{\bot})=-\alpha(a)\},
\end{equation}
that is the set of all maps from the effect algebra $A$ to the reals that satisfy the known property of a measure of union of disjoint subsets, and one that uses negatives. The author of Ref. \cite{Roumen_2017} explored the relation of the elements of $HC^{1}(A)$ and the possible states on $A$, indicating some limitations, like the fact that effect algebras of projections on a Hilbert space don't satisfy the canonical relation between them.

The next step is analogous to the strategy of Čech cohomology on modules, to use relative cohomology to inquiry about the factorizability. One asks for the realization of the effect algebra through inclusion in a Boolean effect algebra $\mathcal{P}(X)$\footnote{A Boolean effect algebra is a Boolean algebra, thus it can be represented as the algebra of subsets of a set $X$, represented as the power set $\mathcal{P}(X)$.}
\begin{equation}
    \begin{tikzcd}
A \arrow{rr}{\sigma} \arrow[hookrightarrow]{dr}{i} & & \left[0,1\right] \\
& \mathcal{P}(X) \arrow{ur}{\eta} &
\end{tikzcd}
\end{equation}
Therefore is the identification of a set $X$ and a state $\eta$, seen as classical, that give the same probabilities of the algebra $A$. The strategy is to use a restriction on the possible effect algebras to the ones that present the canonical relation between states and elements of $HC^{1}(A)$, and with relative cohomology construct a condition to non-factorizability.

\begin{theorem}
Let $i:A\hookrightarrow B$ be a weak injective morphism between finite Archimedean interval effect algebras, and let $\sigma:A\to [0,1]$ be a state. If $\sigma$ extends to a state $\tau:B\to [0,1]$ for which $\tau\circ i=\sigma$, then the cohomology class $\partial (j(\sigma))\in HC^{2}(A,B)$ is zero.
\end{theorem}

Here the map $j$ codifies the canonical relation between states and elements of $HC^{1}(A)$Even for this restriction in the effect algebras, violations due to the use of negatives occur, therefore the theorem is sufficient but not necessary condition. And as the example, the author used the fact that we can extend $\sigma$ to a signed state if and only if the coboundary of a state $\partial (j(\sigma))$ is zero.

\subsubsection*{Order cohomology}

The solution proposed was to restrict the problem to the positive cone of the reals, and hence taking order into account. The cohomology is thus defined in an ordered partial commutative monoid, which all the effect algebras already are.

\begin{definition}
An ordered partial commutative monoid is a partial commutative monoid $A$ equipped with a positive cone $P\subseteq A$, for wihich
\begin{itemize}
    \item $0\in P$;
    \item if $a,b\in P$ and $a\vee b$is defined, then $a\vee b\in P$;
    \item for $a,b\in P$, if $a\vee b=0$, then $a=b=0$.
\end{itemize}
\end{definition}

The map between effect algebras is now understood as maps between positive cones, thus avoiding negatives, and the kernels are extended to precones in the definition of cohomology, which defines commutative monoids $H_{\leq}^{n}(A)$ of an algebra $A$.

\begin{definition}
Let $f:A\to B$ be a morphism between ordered partial commutative monoids. The precone of $f$ is $prec(f)=f^{-1}(B^{+})\subseteq A$.
\end{definition}

This imposition turns order cohomology hard to calculate in relation to cyclic cohomology. Some properties of the cyclic cohomology are lost, as exactness in the relation with relative cohomological monoids, as expected. But it allows writing the first group as
\begin{equation}
    H^{1}_{\leq}\cong\{\alpha:A\to\mathbb{R}_{\geq 0};\alpha(a\vee b)=\alpha(a)+\alpha(b)\},
\end{equation}
without the negative condition, only the condition of measure on disjoint subsets. It includes the possible states, thus it isn't necessary to restrict the kinds of effect algebras as before. 

\begin{theorem}
Let $i:A\hookrightarrow B$ be an injective morphism of effect algebras, and let $\sigma:A\to [0,1]$ be a state. The following are equivalent:
\begin{itemize}
    \item The state $\sigma$ extends to a state $\tau$ on $F$, for which $\tau\circ i=\sigma$.
    \item The state $\sigma$ lies in the precone of the connecting homomorphism $\partial:H_{\leq}^{1}(A)\to H_{\leq}^{2}(B,A)$.
\end{itemize}
\end{theorem}

This result is the characterization of contextual behavior, but codifying the measurement scenario in effect algebras.

\subsubsection*{Relation with sheaf approach}

The method to get the sufficient condition with the use of relative cohomology with coefficients in a ring is similar, but the theorem for cyclic cohomology in effect algebras is more restrictive than one for Čech cohomology in a presheaf. In both frameworks, the violations appear due to the use of negatives. Even the main example of a violation is the same, the equivalence between non-disturbing models (or equivalently effect algebras) and the possibility of use of signal measures. It justifies the use of cohomological tools from sheaf theory in other rings than $\mathbb{R}$, like the Boolean one \cite{abramsky2020classical,abramsky_et_al:LIPIcs:2021:13439}, but in effect algebras only $\mathbb{R}$ is explored.

For the sheaf approach, we start with a semi-ring, thus the canccelative rule can not hold. Such restriction imposes the codification of the obstruction to a pair of objects, that can be calculated to each local context. It is the triviality of all obstructions that implies a non-contextual non-disturbing model. Using outcome determinism one can show that factorizability is the condition to non-contextuality, and non-disturbance implies that the hidden variables are the global sections. Thus contextuality follows from the possibility of extension to a measurement scenario with just one maximal context, and with the global sections identified as the possible events. It also implies the identification of the model as construction from a Boolean effect algebra.

On the other hand, effect algebras can directly be compared to any other via extendability, and in special to Boolean effect algebras. There is no previous construction that says the right effect algebra to use, or even that contextuality follows from the Boolean one. It is the equivalence to the empirical model by the $n$-tests that allows us to talk about contextuality in this framework. 

\begin{proposition}
There is an equivalence between outcome determinism in the sheaf approach and the injective morphism to a Boolean effect algebra in the effect algebra approach.
\end{proposition}

In the sheaf approach, we can impose outcome determinism to ask about contextual behavior. It is equivalent, due to Fine–Abramsky–Brandenburger Theorem, to an injective morphism of the induced effect algebra to a Boolean effect algebra that preserves the measures of the model. As every effect algebra induces a measurement scenario, the proposition follows\footnote{We are omitting the $\lambda$-independence, once it is equivalent to non-disturbance and effect algebras are already non-disturbing, thus the use of Boolean effect algebras is equivalent to outcome determinism.}.

Therefore, the natural way to understand contextuality by the effect algebra is external, by immersion in a chosen algebra, and therefore depends on this choice. This is a signal that a direct treatment for this is categorical, where the existence of such an algebra can be formally described. This contrasts with the sheaf approach, where we need to verify the structure internally, without the need to choose any bigger algebra. 

In the end, the two approaches search the same, to verify if a given structure can be understood as a Boolean structure, or in other words, if it can be understood classically in deterministic ways. Contextuality in the standard sense, deterministic and $\lambda$-independent, is the explanation of a set of $R$-measures via a Boolean logical structure, which is equivalent to factorizability of the measures. And in this sense, both frameworks achieve the cohomological characterization of it.

\section{Commentaries}
\label{Commentaries}

The example of a violation of the Figure \ref{example} is solved by the generalized cohomological contextuality presented. The semi-field is $\mathbb{B}$, and we can't use negatives anymore, thus the only extension of the model for the sections $[(da)\to (01)]$ and $[(da)\to (10)]$ shows a non-trivial obstruction. All the violation examples from \cite{Car__2017,caru2018towards} depend on $"Z"$ structures, therefore they are detected by the obstruction of this paper, as expected by the results presented.

For possibilistic models, where the semi-field is the Boolean $\mathbb{B}$, the contextual behavior is codified explicitly through diagrams, that become very complicated already for three measurements in each context, see Ref. \cite{Beer_2018} and diagrams within as examples. Explore the contextual behavior of possibilistic models, using algebraic or diagrammatic methods, is equivalent to write out the possible extensions of all the sections one can have access to, searching for a non-trivial obstruction.

For other semi-fields, the equivalent calculation for the search of obstruction is the idea behind the contextual fraction \cite{Abramsky_2017}. What one is doing trying to show that all the obstructions are trivial is equivalent to show that the matrix equation $Mg=p$ has a solution. Therefore the contextual fraction is just the canonical way to quantify, for semi-fields that contains the interval $[0,1]$ like $\mathbb{R}$ and $\mathbb{R}^{+}$, the obstruction of the model\footnote{See Ref. \cite{Karvonen_2019} for a functorial view of the contextual fraction}.

\newpage
\section{Conclusion}
\label{Conclusion}

This paper generalized the cohomological contextuality to a characterization of the $R$-contextual behavior for any chosen semi-field $R$. We construct a generalized obstruction, the difference cochain, and with it, we generalized the result of the literature. Violations of the cohomological characterization are absent, since there are no negatives, forbidding $"Z"$ structures of the usual cohomological contextuality.

The result is the Čech cohomological characterization of the search for an extension of all the local sections to global ones. A non-trivial obstruction is equivalent to the impossibility of such extension, and the characterization of $R$-contextuality in the standard sense.

A revision of similar construction, but in the framework of effect algebras, is given. The strategies of both approaches are the same, the use of relative cohomology to compare the model with Boolean structures that preserve the measures of the model. Such Boolean structure is equivalent to the imposition in standard contextuality of the outcome determinism property and can be identified with the set of global events once non-disturbance holds.

A categorical construction such as Ref. \cite{Karvonen_2019}, relating effects algebras and empirical models as in Ref. \cite{Wester_2018}, could be used to explore in a clearer way the different notions of contextuality that don't impose outcome determinism \cite{Spekkens_2005,Spekkens_2014}. Even the use of signed measures to represent non-disturbing models and its relation with non-equivalent negatives that appears in literature could be studied \cite{Spekkens_2008}. And the use of algebraic and logical structures raises the question of the relationship with the topos approach, as in Ref. \cite{Doring2019ContextualityAT}. Finally, the formalization in a unified framework could give more formal tools to explore the relation with operational theories \cite{Selby_2021}, and fundamental questions as Darwinism \cite{Zurek2007RelativeSA,Baldijao2020QuantumDA,Baldijao2021NoncontextualityAA}.

\vspace{22pt}

\appendix\section*{Acknowledgements}

The author thanks the MathFoundQ – UNICAMP – Mathematical Foundations of Quantum Theory, a research group of the Instituto de Matemática, Estatística e Computação Científica, in special Prof. Marcelo Terra Cunha, for the conversations in the preparation of this manuscript. Special thanks to Hamadalli Camargo for reading the draft, and for helpful comments and discussions.

Thanks to the anonymous second reviewer of the 18TH INTERNATIONAL CONFERENCE ON QUANTUM PHYSICS AND LOGIC, for its points in the submission of the first preprint version of this work, in special by citing the cohomology of effect algebras.

Financial support from Brazilian agency Coordenação de Aperfeiçoamento de Pessoal de Nível Superior - CAPES is gratefully acknowledged.

\newpage
\nocite{*}
\bibliographystyle{eptcs}
\bibliography{generic}
\end{document}